\documentclass[letterpaper, 10 pt, conference]{ieeeconf}  % Comment this line out
\IEEEoverridecommandlockouts                              % This command is only
\usepackage{amsmath,amsfonts,amssymb,color}
\usepackage{graphicx}
\usepackage{psfrag}
\usepackage{algorithm}
\usepackage{algpseudocode}
\usepackage{array}
\usepackage{cite}
\usepackage{footmisc}
\usepackage{mathtools}
\usepackage{enumerate}
\usepackage{soul,color}

\newtheorem{theorem}{Theorem}
\newtheorem{corollary}{Corollary}
\newtheorem{lemma}{Lemma}
\newtheorem{remark}{Remark}
\newtheorem{definition}{Definition}
\newtheorem{proposition}{Proposition}
\newtheorem{example}{Example}
\newtheorem{assumption}{Assumption}

\renewcommand{\theenumi}{(\alph{enumi}}

\DeclareMathOperator*{\argmin}{argmin}

\title{\LARGE \bf
Byzantine-Resilient Distributed Optimization of \\ Multi-Dimensional Functions
}

\author{Kananart Kuwaranancharoen, Lei Xin, Shreyas Sundaram
\thanks{This research was supported by NSF CAREER award 1653648. The authors are with the School of Electrical and Computer Engineering at Purdue University.  Email: {\tt\{kkuwaran,lxin,sundara2\}@purdue.edu}. Both of the first two authors contributed equally to this paper.}% <-this % stops a space
}

\begin{document}

\maketitle
\thispagestyle{empty}
\pagestyle{empty}

%%%%%%%%%%%%%%%%%%%%%%%%%%%%%%%%%%%%%%%%%%%%%%%%%%%%%%%%%%%%%%%%%%%%%%%%%%%%%%%%
\begin{abstract}
The problem of distributed optimization requires a group of agents to reach agreement on a parameter that minimizes the average of their local cost functions using information received from their neighbors. While there are a variety of distributed optimization algorithms that can solve this problem, they are typically vulnerable to malicious (or ``Byzantine'') agents that do not follow the algorithm.  Recent attempts to address this issue focus on single dimensional functions, or provide analysis under certain assumptions on the statistical properties of the functions at the agents. In this paper, we propose a resilient distributed optimization algorithm for multi-dimensional convex functions.  Our scheme involves two filtering steps at each iteration of the algorithm: (1) distance-based and (2) component-wise removal of extreme states. We show that this algorithm can mitigate the impact of up to $F$ Byzantine agents in the neighborhood of each regular node, without knowing the identities of the Byzantine agents in advance.  In particular, we show that if the network topology satisfies certain conditions, all of the regular states are guaranteed to asymptotically converge to a bounded region that contains the global minimizer.
\end{abstract}

%%%%%%%%%%%%%%%%%%%%%%%%%%%%%%%%%%%%%%%%%%%%%%%%%%%%%%%%%%%%%%%%%%%%%%%%%%%%%%%%
\section{Introduction} 
The problem of distributed optimization requires a network of agents to reach agreement on a parameter that minimizes the average of their objective functions using local information received from their neighbors. This framework is motivated by various applications including machine learning, power systems, and robotic networks, %. Typically, each agent updates its state by combining the states it received from its neighbors as well as leveraging the gradient of its local function. There exists a
and there are a variety of approaches to solve this problem \cite{nedic2014distributed,xin2019frost,gharesifard2013distributed,zhu2011distributed}.  However, these existing works typically make the assumption that all agents follow the prescribed protocol; indeed, such protocols fail if even a single agent behaves in a malicious or incorrect manner \cite{sundaram2018distributed}.

%However, security plays a central role in distributed system as the distributed manner is vulnerable to component failures. It is shown that the existence of even a single malfunctioning agent can lead the consensual parameter to an arbitrary value, or even prevent the states from reaching consensus \cite{sundaram2018distributed}. Thus, it is of extreme importance to develop algorithms that provide security guarantees. 
A handful of recent papers have considered this problem for the case where agent misbehavior follows prescribed patterns \cite{ravi2019case,wu2018data}. A more general (and serious) form of misbehavior is captured by the {\it Byzantine} adversary model from computer science, where misbehaving agents can send arbitrary (and conflicting) values to their neighbors at each iteration of the algorithm. Under such Byzantine behavior, it has been shown that it is impossible to guarantee computation of the true optimal point \cite{sundaram2018distributed,su2015byzantine}. Thus, some papers have formulated distributed optimization algorithms that allow the non-adversarial nodes to converge to a certain region surrounding the the true minimizer, regardless of the adversaries' actions  \cite{sundaram2018distributed,su2016fault,zhao2017resilient}.  However, the above works focus on single dimensional functions. The extension to general multi-dimensional functions remains largely open, however, since even the region containing the true minimizer of the functions is challenging to characterize in such cases \cite{kuwaranancharoen2018location}. %Directly applying similar techniques used in above papers on each coordinate can only lead to similar results when the problem can be separated into independent scalar-valued problems, which is too restrictive.
%The paper \cite{xu2018robust} considers a problem of robust distributed optimization with time-varying and strongly convex local cost functions, but relies on the careful choice of a regularization parameter, which may be challenging to obtain.  
The recent papers \cite{yang2019byrdie,yang2019Bridge} consider a vector version of the resilient decentralized machine learning problem by utilizing block coordinate descent.  Those papers show that the states of regular nodes will converge to the statistical minimizer with high probability, but the analysis is restricted to i.i.d training data across the network.\footnote{We note that there is also a branch of literature pertaining to optimization in a client-server architecture \cite{gupta2019byzantine} \cite{blanchard2017machine}; this differs from our setting where we consider a fully distributed network of agents.} 

In this paper, we propose an algorithm that extends the ``local-filtering'' dynamics proposed in \cite{sundaram2018distributed,su2016fault} (for single-dimensional convex functions) to the multi-dimensional case. However, this extension is non-trivial, as simply applying the filtering operations proposed in those papers to each coordinate of the parameter vector does not appear to yield clear guarantees.  Instead, we show that by having each node apply an additional filtering step on the parameter vectors that it receives from its neighbors at each step (based on the distance of those vectors from a commonly chosen reference point), one can recover certain performance guarantees in the face of Byzantine adversaries. %Our main contributions are as follows. 1. We propose a consensus-based resilient distributed optimization algorithm for multi-dimensional functions that can mitigate the impact of up to $F$ Byzantine agents in the neighborhood of each regular agent. 2. We provide theoretical guarantees which show that the regular states will asymptotically converge to a bounded region that contains the true minimizer of interest. 
 
%The rest of this paper is organized as follows. We introduce notation in Section~\ref{Notation and Terminology}, and define the problem in Section~\ref{sec:problem formulation}.  In Section~\ref{sec:Algorithm}, we provide our resilient distributed optimization algorithm. In Section~\ref{sec: assumption result}, we  provide our main results, with proofs provided in Section~\ref{sec:Proof}.  Section~\ref{sec:Conclusion and Future work} concludes the paper.

\section{Notation and Terminology} \label{Notation and Terminology}
\subsection{General Notation}
Let $\mathbb R$ and $\mathbb N$ denote the set of real and natural numbers, respectively. For $N \in \mathbb{N}$, let $[N]$ denote the set $\{ 1, 2, \ldots, N \}$. Vectors are taken to be column vectors, unless otherwise noted. We use $[x]_{p}$ to represent the $p$-th component of a vector $x$. The cardinality of a set is denoted by $|\cdot|$, and the Euclidean norm on $\mathbb R^d$ is denoted by $\| \cdot \|_2$. We denote by $\langle u, v \rangle$ the Euclidean inner product of $u$ and $v$ i.e., $\langle u, v \rangle = u^Tv$ and 
by $\angle (u,v)$ the angle between vectors $u$ and $v$ i.e., $\angle (u,v) = \arccos \big( \frac{ \langle u, v \rangle }{\Vert u \Vert_2 \Vert v \Vert_2}  \big)$.
The Euclidean ball in $d$-dimensional space with center at $x_0$ and radius $r$ is denoted by $\mathcal{B}(x_0, r) \triangleq \{ x \in \mathbb{R}^d : \| x - x_0 \|_2 \leq r \} $. Given a convex function $f: \mathbb{R}^d \rightarrow \mathbb{R}$, the set of subgradients of $f$ at any point $x \in \mathbb{R}^d$ is denoted by $\partial f(x)$.

\iffalse
\subsection{Convex Functions}
A set $\mathcal{C} \subseteq \mathbb{R}^d$ is said to be convex if, for all $x, y \in \mathcal{C}$ and all $t \in (0, 1)$, the point $(1- t)x + ty$ also belongs to $\mathcal{C}$. A function $f: \mathbb{R}^d \to \mathbb{R}$ is convex if for all $x \in \mathbb{R}^d$, there exists $g \in \mathbb{R}^d$ such that for all $y \in \mathbb{R}^d$, $f(y) \geq f(x) + \langle g, y-x \rangle$; the vector $g$ is called a subgradient of $f$ at $x$.  
The set of subgradients of $f$ at the point $x$ is called the subdifferential of $f$ at $x$, and is denoted $\partial f(x)$.
\fi

\subsection{Graph Theory}
We denote a network by a directed graph $\mathcal G = ( \mathcal{V}, \mathcal{E})$, which consists of the set of nodes  $\mathcal{V} = \{ v_{1}, v_{2}, \ldots, v_{N} \}$ and the set of edges $\mathcal{E} \subseteq \mathcal{V} \times \mathcal{V}$. If $(v_i, v_j) \in \mathcal{E}$, then node $v_j$ can receive information from node $v_i$. 
The in-neighbor and out-neighbor sets are denoted by $\mathcal{N}^{-}_{i} = \{ v_j \in \mathcal V: \; (v_{j}, v_{i}) \in \mathcal{E} \} $ and $\mathcal{N}^{+}_{i} = \{ v_j \in \mathcal V: \; (v_{i}, v_{j}) \in \mathcal{E} \} $, respectively. A path from node $v_{i}\in \mathcal V$ to node $v_{j}\in \mathcal V$  is a sequence of nodes $v_{k_1},v_{k_2}, \dots, v_{k_l}$ such that $v_{k_1}=v_i$, $v_{k_l}=v_j$ and $(v_{k_r},v_{k_{r+1}}) \in \mathcal{E}$ for $1\leq r \leq l-1$. Throughout the paper, the terms nodes and agents will be used interchangeably.

\begin{definition}
A graph $\mathcal G = ( \mathcal{V}, \mathcal{E})$ is
said to be rooted at node $v_i \in \mathcal{V}$ if for for all nodes $v_j \in \mathcal{V} \setminus \{v_i\}$, there is a path from $v_i$ to $v_j$. A graph is said to be rooted if it is rooted at some node $v_{i}\in \mathcal V$.
\end{definition}

We will rely on the following definitions from \cite{leblanc2013resilient}. 

\begin{definition}[$r$-reachable set]
For any given $r \in \mathbb{N}$, a subset of nodes $\mathcal S  \subseteq \mathcal V$ is said to be $r$-reachable if there exists a node $v_i \in \mathcal S$  such that $|\mathcal N^{-}_{i} \setminus \mathcal S| \geq r$.
\end{definition}

\begin{definition}[$r$-robust graphs] 
For $r \in \mathbb N$, a graph $\mathcal G$ is said to be $r$-robust if for all pairs of disjoint nonempty subsets $\mathcal{S}_1, \mathcal{S}_2 \subset \mathcal V$, at least one of $\mathcal{S}_1$ or $\mathcal{S}_2$ is $r$-reachable.
\end{definition}

The above definitions capture the idea that sets of nodes should contain a node that has a sufficient number of neighbors outside that set. This will be important for the {\it local} decisions made by each node in the network under our algorithm, and will allow information from the rest of the network to penetrate into different sets of nodes.

\subsection{Adversarial Behavior}
\begin{definition}
A node $v_j \in \mathcal V$ is said to be Byzantine if during each iteration of the prescribed algorithm, it is capable of sending arbitrary (and perhaps conflicting) values to different neighbors.
\end{definition}

 The set of Byzantine nodes is denoted by $\mathcal A \subset \mathcal V$. The set of regular nodes is denoted by  $\mathcal{R} = \mathcal{V}\setminus\mathcal{A}$.

The identities of the Byzantine agents are unknown to regular agents in advance. Furthermore, we allow the Byzantine agents know the entire topology of the network and functions equipped by the regular nodes (such worst case behavior is typical in the study of such adversarial models \cite{sundaram2018distributed, su2015byzantine,yang2019byrdie}). 

\begin{definition}[$F$-local model]
For $F \in \mathbb N$, we say that the set of adversaries $\mathcal{A}$ is an $F$-local set if $|\mathcal{N}^{-}_{i} \cap \mathcal{A} | \leq F$, for all $v_{i} \in \mathcal R$.
\end{definition}

Thus, the $F$-local model captures the idea that each regular node has at most $F$ Byzantine in-neighbors.

\section{Problem formulation}\label{sec:problem formulation}
Consider a group of $N$ agents $\mathcal{V}$ interconnected over a graph $\mathcal{G} = (\mathcal{V}, \mathcal{E})$. Each agent $v_i \in \mathcal{V}$ has a local convex cost function $f_i: \mathbb{R}^d \rightarrow \mathbb{R}$.  The objective is to collaboratively solve the following minimization problem:
\begin{align} 
    \mathop{\min}_{x \in \mathbb{R}^d} \frac{1}{N}\sum _{v_i \in \mathcal{V}} f_{i}(x)  \label{prob: min_all_f}
\end{align}
where $x$ is the common decision variable. 
\iffalse
A common approach to solve such problems is to create a local copy of the global decision variable for each node, and impose equality constraints among all local copies; that is,
\begin{align*}
    &\mathop{\min}_{ \{ x_i \} } \frac{1}{N}\sum _{v_i \in \mathcal{V}} f_{i}(x_i) \\
    \text{subject to} \quad &x_1 = x_2 = \ldots = x_N \in \mathbb{R}^d.
\end{align*}
\fi
We assume that nodes can only communicate with their immediate neighbors  to solve the above problem.  However, since Byzantine nodes are allowed to send arbitrary values to their neighbors at each iteration of any algorithm, it is not possible to solve  Problem \eqref{prob: min_all_f} under such misbehavior (since one is not guaranteed to infer any information about the true functions of the Byzantine agents) \cite{sundaram2018distributed, su2015byzantine}.  Thus, the optimization problem is recast into the following form:
\begin{align} 
    \mathop{\min}_{x \in \mathbb{R}^d} \frac{1}{|\mathcal R|}\sum _{v_i \in \mathcal R } f_{i}(x) .  \label{prob: regular node} 
\end{align}
We will now propose an algorithm that allows the regular nodes to approximately solve the above problem (as characterized later in the paper).  We will make the following assumption throughout.

\begin{assumption}
For all $v_i \in \mathcal{V}$, the functions $f_i(x)$ are convex, and the sets $\argmin f_i(x)$ are non-empty and bounded.   \label{ass: convex}
\end{assumption}

\section{A Resilient Distributed Optimization Algorithm} \label{sec:Algorithm}
The algorithm that we propose is stated as Algorithm~\ref{alg: resilient dec opt}. At each time-step $k$, each node $v_i \in \mathcal{V}$ maintains and updates a vector $x_i[k]$, which is its estimate of the solution to Problem~\eqref{prob: regular node}.  After presenting the algorithm, we describe each of the steps and the update rule.

%The analysis of this algorithm is provided in the next section; here we first describe each of the steps of the algorithm.

%Here, we present an algorithm that provides an approximate solution to the problem \eqref{prob: regular node} under some appropriate assumptions. The assumptions and the guarantees of Algorithm \ref{alg: resilient dec opt} will be presented in Section \ref{sec: assumption result}.

\begin{algorithm}
\caption{Distance-MinMax Filtering Dynamics} \label{alg: resilient dec opt}
\textbf{Input} Network $\mathcal{G}$, 
functions $\{f_i\}_{i = 1}^{N}$, the parameter $F$
%\textbf{Output} terminated states $x_i[K] \in \mathbb{R}^d$
\begin{algorithmic}[1]
%\Procedure{MyProcedure}{}
%\BState Define the range of $x^{(1:n)}$
%\BState $X^{(1:n)}$ = Meshgrid($x^{(1:n)}$)
%\For {$v_i \in \mathcal R$}  \Comment{Implement in parallel}
\State Each $v_i \in \mathcal{R}$ sets $x_i^* \gets$ \texttt{optimize}($f_i$) %, \; x_i[0]$)
%\EndFor
\State $\hat{x} \gets$ \texttt{resilient\_consensus} ($F, \; \{ x_i^* \}$)
\State Each $v_i \in \mathcal{R}$ sets $x_i[0] = x_i^*$
\For {$k \in \mathbb{N}$}  
\For {$v_i \in \mathcal R$}  \Comment{Implement in parallel}
\State \texttt{broadcast}($\mathcal{N}_i^+, x_i[k]$)
\State $\mathcal{S}_i[k] \gets$ \texttt{receive}($\mathcal{N}_i^{-}$)
\State $\mathcal{S}_i^{\text{dist}}[k] \gets$ \texttt{dist\_filter}($F, \; \hat{x}, \; \mathcal{S}_i[k]$)
\State $\mathcal{S}_i^{\text{mm}}[k] \gets$ \texttt{minmax\_filter}($F ,\; \mathcal{S}_i^{\text{dist}}[k]$)
\State $z_i[k] \gets$ \texttt{average}($ \mathcal{S}_i^{\text{mm}}[k]$)
\State $x_i[k+1] \gets$ \texttt{gradient}($f_i, z_i[k]$)
\EndFor
\EndFor
%\State \Return $x_i[K]$ for $i \in [N]$
%\EndProcedure
\end{algorithmic}
\end{algorithm}
Note that Byzantine nodes do not necessarily need to follow the above algorithm, and can update their states however they wish.  We now explain each step used in  Algorithm \ref{alg: resilient dec opt}.
\begin{enumerate}[1.]
    \item $x_i^* \gets$ \texttt{optimize} ($f_i$) %, \; x_i[0]$) 
    \\
    Each node $v_{i} \in \mathcal R$ finds the minimizer $x_i^*$ of its local function $f_i$ (using any appropriate algorithm).
    
    \item $\hat{x} \gets$ \texttt{resilient\_consensus} ($F, \; \{ x_i^* \}$) \\
    The nodes run a resilient consensus algorithm to calculate a consensus point (which we term an {\it auxiliary point}) $\hat{x} \in \mathbb{R}^d$, with each node $v_i \in \mathcal{R}$ setting its initial vector to be its individual minimizer $x_i^*$. For example, $d$ parallel versions of the the resilient scalar consensus algorithm from \cite{leblanc2013resilient} can be applied (one for each component of the parameter vector).  This is guaranteed to return a consensus value $\hat{x}$ that is in the smallest hyperrectangle containing all of the minimizers of the regular nodes' functions, regardless of the actions of any $F$-local set of adversaries (under the network conditions provided in the next section) \cite{leblanc2013resilient}.
    %Node $v_{i} \in \mathcal R$ updates its auxiliary point $\hat{x}_i[k]$ using a resilient consensus algorithm that can tolerate up to $F$-local Byzantine agents by using $\{ x_i^* : v_i \in \mathcal R\}$ as an initialization, e.g. applying Algorithm 1 in \cite{leblanc2013resilient} on each coordinate.
    
    \item \texttt{broadcast} ($\mathcal{N}_i^+, x_i[k]$) \\
    Node $v_{i} \in \mathcal{R}$ broadcasts its current state $x_{i}[k]$ to its out-neighbors $\mathcal{N}_i^+$.
    
    \item $\mathcal{S}_i[k] \gets$ \texttt{receive}($\mathcal{N}_i^{-}$) \\
     Node $v_{i} \in \mathcal{R}$ receives the current states from its in-neighbors $\mathcal{N}_i^-$. So, at time step $k$, node $v_i$ possesses the set of states $\mathcal{S}_i[k] \triangleq \{ x_j[k]: j \in \mathcal{N}_i^- \} \cup \{ x_i[k] \}$.
    
    \item $\mathcal{S}_i^{\text{dist}}[k] \gets$ \texttt{dist\_filter} ($F, \; \hat{x}, \; \mathcal{S}_i[k]$) \\
    Node $v_i \in \mathcal R$ computes 
    \begin{align}
        D_{ij}[k] \triangleq \| x_j[k] - \hat{x} \|_2  \;\; \text{for} \;\; x_j[k] \in \mathcal{S}_i[k] \setminus \{ x_i[k] \}.    \label{def: D-metric}
    \end{align}
    Then, node $v_i$ removes the states $x_j[k] \in \mathcal{S}_i[k] \setminus \{ x_i[k] \}$ that produce the $F$-highest values of $D_{ij}[k]$. The remaining states $x_j[k]$ are stored in $\mathcal{S}_i^{\text{dist}}[k]$.
    
    \item $\mathcal{S}_i^{\text{mm}}[k] \gets$ \texttt{minmax\_filter} ($F ,\; \mathcal{S}_i^{\text{dist}}[k]$) \\
    Node $v_i$ further removes states that have extreme values in any of their components.  More specifically, node $v_i \in \mathcal R$ removes the state $x_j[k] \in \mathcal{S}_i^{\text{dist}}[k] \setminus \{ x_i[k] \}$ if there exists $p \in [d]$ such that $[ x_j[k] ]_p$ is in the $F$-highest or $F$-lowest values from the set of scalars $\big\{ [ x_l[k] ]_p : x_l[k] \in \mathcal{S}_i^{\text{dist}}[k] \setminus \{ x_i[k] \} \big\}$. The remaining states are stored in $\mathcal{S}_i^{\text{mm}}[k]$.
    
    \item $z_i[k] \gets$ \texttt{average} ($\mathcal{S}_i^{\text{mm}}[k]$) \\
    %For $v_i \in \mathcal R$, the stochastic vector\footnote{SS: can we just use the average here to simplify?} $w_i[k] \in \mathbb{R}^{ | \mathcal{S}_i^{\text{mm}}[k] | }$ with positive entries is chosen by node $v_i$.
    Each node $v_i \in \mathcal{R}$ computes
    \begin{equation}
        z_i[k] = \sum_{v_j \in \mathcal{N}_i^{\text{mm}} [k] } \frac{1}{|\mathcal{S}_i^{\text{mm}} [k]|} x_j[k]  \label{eq: z_i_calc}%[ w_i[k] ]_j  x_j[k]
    \end{equation}
    where $\mathcal{N}_i^{\text{mm}} [k] \triangleq \{ v_s : x_s[k] \in \mathcal{S}_i^{\text{mm}} [k] \}$.
    
    \item $x_i[k+1] \gets$ \texttt{gradient} ($f_i, z_i[k]$) \\
    Node $v_i \in \mathcal{R}$ computes the gradient update as follows:
    \begin{align}
        x_i[k+1] = z_i[k] - \eta[k] g_i[k]  \label{eqn: dynamic}
    \end{align}
    where $g_i[k] \in \partial f_i( z_i[k] )$ and $\eta[k]$ is the step-size at time-step $k$.
\end{enumerate}

\begin{remark}
The role of the auxiliary point $\hat{x}$ is to give the states $x_i[k]$ a ``sense of good direction''.  In other words, since adversarial nodes can try to pull the regular nodes away from the true minimizer, the auxiliary point provides a common reference point for each regular node with which to evaluate the states in its neighborhood.  This motivates the distance-based filtering step of Algorithm~\ref{alg: resilient dec opt} (in line 8), which removes the $F$ states that are furthest away from the auxiliary point at each time-step.  Note that the auxiliary point will, in general, be different from the true minimizer of interest.  We also note that the resilient consensus algorithm from \cite{leblanc2013resilient} for computing the auxiliary point will only provide asymptotic consensus in general.  The regular agents will converge to consensus exponentially fast under that algorithm, however, and thus we expect that running the resilient consensus algorithm to compute the auxiliary point simultaneously with the other update steps (given by lines 6-11) will also lead to the same guarantees.  In the interest of space, we do not provide that analysis here.
\end{remark}

%While our proof of convergence in the following section will work for any arbitrary $\hat{x}\in \mathbb{R}^d$, if we choose the auxiliary point to be far from the set of minimizers, we will get a bad guarantee result. There are several ways to choose an auxiliary point. In an environment where we have a trusted agent, we can pre-install this minimizer as an auxiliary point on all agents. If there is no trusted agent, we can apply the step described in the above algorithm. The dynamics of the auxiliary point is purely based on convex combinations and can be shown to reach consensus on each coordinate $p$ exponentially fast. We only present the proof based on a fixed $\hat x$ in this paper for the sake of simplicity.

\section{Assumptions and Main Results}  \label{sec: assumption result}

We will make the following assumptions in our analysis.

\begin{assumption}
There exists $L > 0$ such that $\| g_i(x) \|_2 \leq L$ for all $x \in \mathbb{R}^d$ and $v_i \in \mathcal V$, where $g_i(x) \in \partial f_i(x)$.   \label{ass: gradient_bound}
\end{assumption}

\begin{assumption}
The step-sizes used in line 11 of Algorithm~\ref{alg: resilient dec opt} satisfy $\lim_{k \to \infty} \eta[k] = 0$, $\eta[k+1] < \eta[k]$ for all $k$, and $\sum_{k=0}^{\infty} \eta [k] = \infty$.   \label{ass: step-size}
\end{assumption}

\begin{assumption} \label{ass: robust}
The underlying communication graph $\mathcal{G}$ is $\big( (2d+1)F+1 \big)$-robust, and the Byzantine agents form a $F$-local set. 
\end{assumption}

We will also use the following lemma from \cite{leblanc2013resilient}.

\begin{lemma} \label{lem: root}
Suppose a graph $\mathcal G$ satisfies Assumption~\ref{ass: robust}. Let $\mathcal G'$ be a graph obtained by removing $(2d + 1)F$ or fewer incoming edges from each node in $\mathcal G$. Then $\mathcal G'$ is rooted.
\end{lemma}

\subsection{Convergence to Consensus}
We first show that the states $x_i[k]$ of all regular nodes $v_i \in \mathcal{R}$ reach consensus under Algorithm~\ref{alg: resilient dec opt}.  

\begin{theorem}[Consensus]
Under Assumptions \ref{ass: gradient_bound}, \ref{ass: step-size}, and \ref{ass: robust}, $\lim_{k\rightarrow\infty}\|x_{i}[k]-x_{j}[k]\| = 0$ for all $v_i, v_j \in \mathcal{R}$.
\label{thm:consensus}
\end{theorem}

\iffalse
$\forall i\in \mathcal R$, $\forall p\in [d]$, there exists $q^{(p)} \in \mathbb{R}$ such that $\lim_{k\to\infty} [ x_i[k]]_p= q^{(p)}$.\footnote{SS: Is this true?  This states that all nodes converge to a {\it constant} value $q^{(p)}$, which may not happen with adversarial nodes (as shown for scalar optimization in my paper).  For consensus, need $\|x_i[k] - x_j[k]\| \rightarrow 0$ for all $v_i, v_j \in \mathcal{R}$.  Please verify that this is what you mean. (LX:The Theorem 1 in ByRDie has a time-dependent term $m$ in the consensual value. This $m$ is carefully defined in their paper.}
\fi

\begin{proof}
We will argue that all regular nodes $v_i \in \mathcal{R}$ reach consensus on each component of their vectors $x_i[k]$, which will then prove the result.  For all $p \in [d]$ and for all $v_i \in \mathcal{R}$,  from \eqref{eqn: dynamic}, the $p$-th component of the vector $x_i[k]$ evolves as
\begin{equation}
[x_i[k+1]]_p = [z_i[k]]_p - \eta[k][g_i[k]]_p.
\label{eq:x_component}
\end{equation}
From \eqref{eq: z_i_calc}, the quantity $z_i[k]$ is an average of a subset of the parameter vectors from node $v_i$'s neighborhood.  In particular, the set $\mathcal{S}_i^{mm}[k]$ is obtained by removing at most $(2d+1)F$ of the vectors received from $v_i$'s neighbors ($F$ vectors removed by the distance based filtering in line 8, and up to $2F$ additional vectors removed by the minmax filtering step on each of the $d$ components in line 9 of the algorithm).  Thus, at each time-step $k$, component $[z_i[k]]_p$ is an average of at least $|\mathcal{N}_i^{-}| - (2d+1)F$ of $v_i$'s neighbors values on that component.  Since the graph is $((2d+1)F + 1)$-robust and the Byzantine agents form an $F$-local set (by Assumption~\ref{ass: robust}), and leveraging Lemma~\ref{lem: root} and the fact that the term $\eta[k]g_i[k]$ asymptotically goes to zero (by Assumptions~\ref{ass: gradient_bound} and \ref{ass: step-size}), we can use an identical argument as in Theorem 6.1 from \cite{sundaram2018distributed} to show that the scalar dynamics \eqref{eq:x_component} converge to consensus.  
\end{proof}

\iffalse
The filters (\texttt{dist\_filter} and \texttt{minmax\_filter}) applied by each node $v_i\in \mathcal R$ on the states ${x}_j[k]$ for $v_j\in \mathcal{N}^{-}_{i}$at time step $k$ will remove at most $(2d+1)F$ scalars in  total on each coordinate $p$. By Assumption \ref{ass: robust}, the $p$'th component of each Byzantine state $[{x}_{j}[k]]_p$ for $j \in \mathcal N^{-}_{i}\cap \mathcal A$ will either be removed by node $i$ or can be expressed as some convex combination of regular values on that coordinate at time step $k$.
Based on Lemma~\ref{lem: root}, the proof of the following theorem can be shown \cite[Theorem~1]{yang2019byrdie}.

The proof of the above Theorem mainly used the fact that the product of the matrices which specify convex combination rules will converge if each of them is row stochastic at each time step. The effect of the extra gradient term will go to zero if the step size $\eta [k]$ is decreasing to zero as time goes to infinity. Hence, the states of regular nodes will asymptotically reach consensus on each coordinate $p$ as time goes to infinity.
\fi

\subsection{What Region Do the States Converge To?}

We now analyze the trajectories of the states of the agents under Algorithm~\ref{alg: resilient dec opt}.  %Recall that $\hat{x} \in \mathbb{R}^d$ is the auxiliary point assumed to be a constant among all the nodes. We then state the result from the filtering update (\texttt{dist\_filter}, \texttt{minmax\_filter}, and \texttt{average}) presented in Algorithm \ref{alg: resilient dec opt}.
We start with the following result about the quantity $z_i[k]$ calculated in line 10 of the algorithm (and given by equation~\eqref{eq: z_i_calc}). 

\begin{proposition} \label{prop: alg prop}
%Suppose that for all $k \in \mathbb{N}$, the set of states $\{ x_i[k]: v_i \in \mathcal{R} \}$ is updated to become $\{ z_i[k]: v_i \in \mathcal{R} \}$ using the methods \texttt{dist\_filter}, \texttt{minmax\_filter} and \texttt{average} in Algorithm \ref{alg: resilient dec opt}.
For all $k \in \mathbb{N}$ and $v_i \in \mathcal{R}$, if there exists $R_i[k] \in \mathbb{R}_{\ge 0}$ such that $ \| x_j[k] - \hat{x} \|_2 \leq R_i[k]$ for all $v_j \in (\mathcal{N}_i^- \cap \mathcal{R}) \cup \{ v_i \}$ then $ \| z_i[k] - \hat{x} \|_2 \leq R_i[k]$.
\end{proposition}  
    
\begin{proof}
Consider the step $\mathcal{S}_i^{\text{dist}}[k] \gets$ \texttt{dist\_filter} ($F, \; \hat{x}, \; \mathcal{S}_i[k]$) in Algorithm \ref{alg: resilient dec opt}. We will first prove the following claim. For each $v_i \in \mathcal{R}$, there exists $v_r \in ( \mathcal{N}_i^- \cap \mathcal{R} ) \cup \{ v_i \}$ such that $\| x_j[k] - \hat{x} \|_2 \leq \| x_r[k] - \hat{x} \|_2$ for all $v_j \in \{ v_s: x_s[k] \in \mathcal{S}_i^{\text{dist}}[k] \}$.

There are two possible cases.  First, if the set $\mathcal{S}_i^{\text{dist}}[k]$ contains only regular nodes, we can simply choose $v_r \in ( \mathcal{N}_i^- \cap \mathcal{R} ) \cup \{ v_i \}$ to be the node whose state $x_r[k]$ is furthest away from $\hat{x}$.  Next, consider the case where $\mathcal{S}_i^{\text{dist}}[k]$ contains the states of one or more Byzantine nodes. Since node $v_i \in \mathcal{R}$ removes the $F$ states from $\mathcal{N}_{i}^{-}$ that are furthest away from $\hat{x}$ (in line 8 of the algorithm), %have the $F$-highest values of $D_{ij}[k]$ (where $D_{ij}[k]$ is defined in \eqref{def: D-metric}) 
and there are at most $F$ Byzantine nodes in $\mathcal{N}_i^-$, there is at least one regular state removed by the node $v_i$. Let $v_r$ be one of the regular nodes whose state is removed.  We then  have $D_{ir} [k] \geq D_{ij} [k]$, for all $v_j \in \{ v_s: x_s[k] \in \mathcal{S}_i^{\text{dist}}[k] \}$ which proves the claim.  

Consider the step $\mathcal{S}_i^{\text{mm}}[k] \gets$ \texttt{minmax\_filter} ($F ,\; \mathcal{S}_i^{\text{dist}}[k]$) in Algorithm~\ref{alg: resilient dec opt}. We have that $\mathcal{S}_i^{\text{mm}}[k] \subset \mathcal{S}_i^{\text{dist}}[k]$. Then, consider the step $z_i[k] \gets$ \texttt{average} ($\mathcal{S}_i^{\text{mm}}[k]$). We have
\begin{equation*}
    z_i[k] - \hat{x} 
    = \sum_{ v_j \in \mathcal{N}_i^{\text{mm}} [k]}  \frac{1}{|\mathcal{S}_i^{\text{mm}} [k]|}   \big( x_j[k] - \hat{x} \big) .
\end{equation*}
Since $\| x_j[k] - \hat{x} \|_2 \leq \| x_r[k] - \hat{x} \|_2$ for all $v_j \in \mathcal{N}_i^{\text{mm}} [k]$ (where $v_r$ is the node identified in the claim at the start of the proof), we obtain
\begin{equation*}
    \| z_i[k] - \hat{x} \|_2 
    \leq \sum_{ v_j \in \mathcal{N}_i^{\text{mm}} [k]}  \frac{1}{|\mathcal{S}_i^{\text{mm}} [k]|}   \| x_j[k] - \hat{x} \|_2 
    %&\leq \sum_{ v_j \in \mathcal{N}_i^{\text{mm}} [k] \cap \mathcal{R} }  \frac{1}{|\mathcal{S}_i^{\text{mm}} [k]|}   R_i[k] \\
    %& \qquad \qquad + \sum_{ v_j \in \mathcal{N}_i^{\text{mm}} [k] \cap \mathcal{A} }  \frac{1}{|\mathcal{S}_i^{\text{mm}} [k]|}   R_i[k] \\
     \le R_i[k].
\end{equation*}
\end{proof}

Next, we will establish certain quantities that will be useful for our analysis of the convergence region.  Since the set $\argmin f_i(x)$ is  non-empty for all $v_i \in \mathcal{V}$ (by Assumption~\ref{ass: convex}), we define $x_i^* \in \mathbb{R}^d$ to be a minimizer of $f_i(x)$ for all $v_i \in \mathcal{V}$. For $\epsilon > 0$, define
\begin{align}
    \mathcal{C}_i (\epsilon) \triangleq \{ x \in \mathbb{R}^d : f_i(x) \leq f_i(x_i^*) + \epsilon \}.   \label{def: sublevel_set}
\end{align}
Since the set $\argmin f_i(x)$ is bounded for all $v_i \in \mathcal{V}$, there exists $\delta_i (\epsilon) \in (0, \infty)$ such that $\mathcal{C}_i( \epsilon ) \subseteq \mathcal{B}(x_i^*, \delta_i (\epsilon))$ for all $v_i \in \mathcal{V}$. 

%Typically, for any convex functions $f$, we have that for any $x$, $\angle (-g_i(x), x_i^* - x) \leq \frac{\pi}{2}$. However, by using Assumption \ref{ass: gradient_bound} and Assumption \ref{ass: convex} together, we can get a better bound for the angle as we will formally state in the following proposition.

\begin{proposition} \label{prop: grad}
Consider a  convex function $f$ that has bounded subgradients, and suppose the set $\argmin f(x)$ is non-empty and bounded.  Then for all $\epsilon > 0$ there exists $\theta( \epsilon ) \ge 0$ such that $\angle ( - g(x), \; x^* - x ) \leq \theta (\epsilon) < \frac{\pi}{2}$ for all $x \notin \mathcal{C}( \epsilon )$, $g(x) \in \partial f(x)$, and $x^* \in \argmin f(x)$, where $\mathcal{C}( \epsilon )$ is defined in the same way as \eqref{def: sublevel_set}.  
\end{proposition}

\begin{proof}
From the definition of convex functions, for any $x$, $y\in \mathbb{R}^d$, we have $f(y) \geq f(x) + \langle g(x), \; y - x \rangle$, where $g(x) \in \partial f(x)$. Substitute a minimizer $x^*$ of the function $f$ into the variable $y$ to get
\begin{align}
    - \langle g(x), \; x^* - x \rangle \geq f(x) - f(x^*).  \label{eqn: convex ineq}
\end{align}
Let $\hat{\theta}(x) \triangleq \angle ( g(x), x - x^* )$.
The inequality \eqref{eqn: convex ineq} becomes
\begin{align*}
    \| g(x) \|_2 \; \| x^* - x \|_2 \cos \hat{\theta}(x) &\geq f(x) - f(x^*).
\end{align*}
Fix $\epsilon > 0$, and  suppose that $x \notin \mathcal{C}(\epsilon)$. Applying $\| g(x) \|_2 \leq L$, we have
\begin{align}
    \cos \hat{\theta}(x) \geq \frac{f(x) - f(x^*)}{\| g(x) \|_2 \; \| x^* - x \|_2} \geq \frac{f(x) - f(x^*)}{L \; \| x^* - x \|_2}.  \label{eqn: cos theta}
\end{align}
Let $\Tilde{x}$ be the point on the line connecting $x^*$ and $x$ such that $f( \Tilde{x} ) = f( x^* ) + \epsilon$. We can rewrite the point $x$ as
\begin{align*}
    x = x^* + t ( \Tilde{x} - x^* )  \quad \text{where} \quad t = \frac{ \| x - x^* \|_2 }{ \| \Tilde{x} - x^* \|_2 } \geq 1.
\end{align*}
Suppose $\mathcal{C}( \epsilon ) \subseteq \mathcal{B}(x^*, \delta (\epsilon))$. Consider the term on the RHS of \eqref{eqn: cos theta}. We have
\begin{align}
    \frac{ f(x) - f(x^*) }{ \| x - x^* \|_2 } &= \frac{f( x^* + t ( \Tilde{x}- x^*) ) - f(x^*) }{ t \| \Tilde{x} - x^* \|_2}  \nonumber \\
    &\geq  \frac{f( x^* + t ( \Tilde{x}- x^*) ) - f(x^*) }{ t \max_{x \in \mathcal{C}(\epsilon)} \| x - x^* \|_2} \nonumber \\
    &\geq  \frac{f( x^* + t ( \Tilde{x}- x^*) ) - f(x^*) }{ t \delta (\epsilon)}.  \label{eqn: t increase}
\end{align}
Since the quantity $\frac{f( x^* + t ( \Tilde{x}- x^*) ) - f(x^*) }{ t }$
is non-decreasing in $t \in (0, \infty)$ \cite[Lemma~2.80]{mordukhovich2013easy}, the inequality \eqref{eqn: t increase} becomes
\begin{align}
    \frac{ f(x) - f(x^*) }{ \| x - x^* \|_2 } \geq \frac{ f(\Tilde{x}) - f(x^*)}{\delta (\epsilon)}
    = \frac{\epsilon}{ \delta(\epsilon)}.  \label{eqn: lb_grad}
\end{align}
Therefore, combining \eqref{eqn: cos theta} and \eqref{eqn: lb_grad}, we obtain
\begin{align}
    \cos \hat{\theta}(x) \geq  \frac{\epsilon}{L \delta(\epsilon)}.  \label{eqn: cos theta hat}
\end{align}
However, from the definition of convex functions, we have 
\begin{align*}
    f(x^*) \geq f( \Tilde{x} )+ \langle g( \Tilde{x} ), \; x^* - \Tilde{x} \rangle.
\end{align*}
Since $\|  g( \Tilde{x} ) \|_2 \leq L$ and $\| x^* - \Tilde{x} \|_2 \leq \delta(\epsilon)$, we get 
\begin{align*}
    \epsilon =  f( \Tilde{x} ) - f(x^*) \leq - \langle g( \Tilde{x} ), \; x^* - \Tilde{x} \rangle \leq L \delta(\epsilon).
\end{align*}
Since $\epsilon > 0$ and $\frac{\epsilon}{L \delta(\epsilon)} \leq 1$ from the above inequality, the inequality \eqref{eqn: cos theta hat} becomes
\begin{align*}
    \hat{\theta}(x) \leq \arccos \Big( \frac{\epsilon}{L \delta(\epsilon)} \Big) \triangleq \theta(\epsilon) < \frac{\pi}{2}.
\end{align*}
\end{proof}

%The proof of Proposition~\ref{prop: grad} is given in the Appendix.  
From Proposition \ref{prop: grad}, if   $f_i$ satisfies Assumptions~\ref{ass: convex} and \ref{ass: gradient_bound} for all $v_i \in \mathcal{R}$, then for all $\epsilon > 0$, we have $\angle ( - g_i(x), \; x_i^* - x ) \leq \theta_i (\epsilon) < \frac{\pi}{2}$ for all $x \notin \mathcal{C}_i( \epsilon )$ and $v_i \in \mathcal{R}$.

Define $\Tilde{R}_i \triangleq \| x_i^* - \hat{x} \|_2$ and
\begin{align}
    R^* \triangleq \inf_{\epsilon > 0} \Big\{ \max_{v_i \in \mathcal{R}} \big\{ \max \{ \Tilde{R}_i \sec \theta_i(\epsilon) , \; \Tilde{R}_i + \delta_i(\epsilon) \} \big\} \Big\}.  \label{def: R^*}
\end{align}
We now come to the main result of this paper, showing that the states of all the regular nodes will asymptotically converge to a ball of radius $R^*$ around the auxiliary point $\hat{x}$ under Algorithm~\ref{alg: resilient dec opt}. 

\begin{theorem}[Convergence] \label{thm: main}
If the states are updated using Algorithm~\ref{alg: resilient dec opt}, and Assumptions~\ref{ass: convex}, \ref{ass: gradient_bound}, \ref{ass: step-size} and \ref{ass: robust} hold, then for all $v_i \in \mathcal{R}$, $\limsup_k \| x_i[k] - \hat{x} \|_2 \leq R^*$, regardless of the actions of any $F$-local set of Byzantine adversaries.
\end{theorem}

The full proof of the theorem is somewhat long, and is thus provided in the Appendix. We present here a sketch of the proof. We work towards the proof of Theorem \ref{thm: main} in several steps.  
For a fixed $\epsilon > 0$ and for all $\xi \in \mathbb{R}_{> 0}$, define the \textit{convergence radius}
\begin{align*}
    s^*( \xi, \epsilon ) \triangleq \max_{v_i \in \mathcal{R}} \big\{ \max \{ \Tilde{R}_i \sec \theta_i(\epsilon), \Tilde{R}_i + \delta_i(\epsilon) \} \big\} + \xi.
\end{align*}
Throughout the paper, we will omit the dependence of $\epsilon$ in $s^*( \xi, \epsilon )$ for notational simplicity.

The key idea is that for a fixed $\xi > 0$, we partition the space $\mathbb{R}^d$ into $3$ regions: 
\begin{enumerate}[1.]
    \item $\mathcal{B} (\hat{x}, \; \max_{v_i \in \mathcal{R}} \{ \Tilde{R}_i + \delta_i \} )$,   \label{rgn: 1}
    \item $\mathcal{B} (\hat{x}, \; s^*(\xi) ) \setminus \mathcal{B} (\hat{x}, \; \max_{v_i \in \mathcal{R}} \{ \Tilde{R}_i + \delta_i \} )$, and   \label{rgn: 2}
    \item $\mathbb{R}^d \setminus \mathcal{B} (\hat{x}, \; s^*(\xi) )$.  \label{rgn: 3}
\end{enumerate}

We refer to $\mathcal{B} (\hat{x}, \; s^*(\xi) )$ as the \textit{convergence region} and note that $R^* = \inf_{ \epsilon > 0 } s^*(0)$. In Steps 1, 2 and 3, we analyze the gradient update \eqref{eqn: dynamic} for each regular agent $v_i \in \mathcal{R}$, and in Step 4, we consider the sequence of distance updates of the regular agent furthest from the auxiliary point. Finally, we will take $\xi$ to go to zero and take the infimum of the distance bound over $\epsilon > 0$.

\textbf{Step 1:} First, we show that if $k$ is sufficiently large and the state $z_i[k]$ is in Region 1 i.e., 
\begin{align*}
    z_i[k] \in \mathcal{B} (\hat{x}, \max_{v_j \in \mathcal{R}} \{ \Tilde{R}_j + \delta_j \} ) \subset \mathcal{B} (\hat{x}, s^*( \xi ) ),
\end{align*}
then by applying the gradient update \eqref{eqn: dynamic}, the state $x_i[k+1]$ is still in the convergence region i.e., $x_i[k+1] \in \mathcal{B} (\hat{x}, s^*( \xi ) )$. To do this, we leverage the fact that the magnitude of the gradient step satisfies $\eta[k] \| g_i[k] \|_2 \leq \eta[k] L$ by Assumption \ref{ass: gradient_bound} and $\eta[k] L$ decreases as $k$ increases by Assumption \ref{ass: step-size}.

\textbf{Step 2:} We find the relationship between the terms $\| x_i[k+1] - \hat{x} \|_2$ and $\| z_i[k] - \hat{x} \|_2$ which will be used in the subsequent step. Specifically, for $v_i \in \mathcal{R}$, if the state $z_i[k]$ is in region 2 or 3, i.e., $\| z_i[k] - \hat{x} \|_2 > \max_{v_j \in \mathcal{R}} \{ \Tilde{R}_j + \delta_j \}$, then we show 
\begin{multline}
    \| x_i[k+1] - \hat{x} \|_2^2 \leq \| z_i[k] - \hat{x} \|_2^2  \\
    - \Delta_i( \| z_i[k] - \hat{x} \|_2, \; \eta[k] \; \| g_i[k] \|_2 )   \label{eqn: z-x general update}
\end{multline} 
where we define $\Delta_i: [ \Tilde{R}_i, \infty) \times \mathbb{R}_+ \to \mathbb{R}$ to be the function
\begin{align*}
    \Delta_i( p, l ) \triangleq 2 l  \Big( \sqrt{ p^2 - \Tilde{R}_i^2 } \cos \theta_i - \Tilde{R}_i \sin \theta_i \Big) - l^2.  
\end{align*}

\textbf{Step 3:} We show that if $k$ is sufficiently large and the state $z_i[k]$ is in Region 2 i.e., $\| z_i[k] - \hat{x} \|_2 \in \big( \max_{v_j \in \mathcal{R}} \{ \Tilde{R}_j + \delta_j \}, s^*(\xi) \big]$ then by applying the gradient update \eqref{eqn: dynamic}, we have that the state $x_i[k+1]$ is still in the convergence region i.e., $x_i[k+1] \in \mathcal{B} (\hat{x}, s^*( \xi ) )$. This is done by showing that small $\eta[k] L$ makes the RHS of \eqref{eqn: z-x general update} bounded above by $\big( s^*(\xi) \big)^2$.

From Step 1 and 3, and Proposition~\ref{prop: alg prop}, we can conclude that for sufficiently large $k$, for each regular node $v_i \in \mathcal{R}$, if the state $x_i[k]$ is in the convergence region, then the state $x_i[k+1]$ is still in the convergence region, i.e., the regular states cannot leave the convergence region.

\textbf{Step 4:} We show that the states of all regular nodes eventually enter the convergence region. Specifically, from Proposition~\ref{prop: alg prop}, we have that 
\begin{align}
    \max_{v_i \in \mathcal{R}} \| z_i[k] - \hat{x} \|_2 \leq \max_{v_i \in \mathcal{R}} \| x_i[k] - \hat{x} \|_2.  \label{eqn: max ineq}
\end{align}
On the other hand, for the agents whose state $z_i[k]$ is in Region 3, we show that the term $\Delta_i( \| z_i[k] - \hat{x} \|_2, \; \eta[k] \; \| g_i[k] \|_2 )$ in \eqref{eqn: z-x general update} is bounded below  as 
\begin{align}
    \Delta_i( \| z_i[k] - \hat{x} \|_2, \eta[k] \; \| g_i[k] \|_2 ) > c_i \eta[k]  \label{eqn: grad_chain}
\end{align}
for sufficiently large $k$,
where $c_i$ is a positive constant. Incorporating \eqref{eqn: max ineq} and \eqref{eqn: grad_chain} into \eqref{eqn: z-x general update}, we can conclude that the quantity $\max_{v_i \in \mathcal{R}} \| x_i[k] - \hat{x} \|_2^2$ strictly decreases if it is greater than $s^*(\xi)$. Eventually, every state will  enter the convergence region, which completes the proof.

To gain insight into the convergence region, we provide the following result.

\begin{proposition} \label{prop: true_sol}
Let $x^*$ be a solution of Problem~\eqref{prob: regular node}. If Assumptions~\ref{ass: convex} and \ref{ass: gradient_bound} hold, then $x^* \in \mathcal{B} (\hat{x}, R^*)$ where $R^*$ is defined in \eqref{def: R^*}.
\end{proposition} 

\begin{proof}
We will show that the summation of any subgradients of the regular nodes' functions at any point outside the region $\mathcal{B} (\hat{x}, R^*)$ cannot be zero. 

Let $x_0$ be a point outside $\mathcal{B} (\hat{x}, R^*)$. Since $\| x_0 - \hat{x} \|_2 > \max_{v_i \in \mathcal{R}} \{ \Tilde{R}_i + \delta_i(\epsilon) \}$ for some $\epsilon > 0$, we have that $x_0 \notin \mathcal{C}_i(\epsilon)$ for all $v_i \in \mathcal{R}$. By the definition of $\mathcal{C}_i(\epsilon)$ in \eqref{def: sublevel_set}, we have 
$f_i(x_0) > f_i(x_i^*) + \epsilon$ for all $v_i \in \mathcal{R}$. Since the functions $f_i$ are convex, we obtain $g_i(x_0) \neq \mathbf{0}$ for all $v_i \in \mathcal{R}$ where $g_i(x_0) \in \partial f_i(x_0)$.

Consider the angle between the vectors $x_0 - x_i^*$ and $x_0 - \hat{x}$. Suppose $\Tilde{R}_i > 0$; otherwise, we have $\angle (x_0 - x_i^*, \; x_0 - \hat{x}) = 0$. Using Lemma~\ref{lem: max_angle} in the Appendix, we can bound the angle as follows: %\footnote{This geometric property for higher dimensions can be proved using coordinate transformation.}: 
\begin{align*}
    \angle (x_0 - x_i^*, \; x_0 - \hat{x}) 
    &\leq \max_{y \in \mathcal{B}(\hat{x}, \Tilde{R}_i)}  \angle (x_0 - y, \; x_0 - \hat{x})  \\
    &= \arcsin \Big( \frac{ \Tilde{R}_i }{ \| x_0 - \hat{x} \|_2} \Big). 
\end{align*}
Since $\| x_0 - \hat{x} \|_2 > \max_{v_i \in \mathcal{R}} \{ \Tilde{R}_i \sec \theta_i  \}$ and $\arcsin(x)$ is an increasing function in $x \in [-1, 1]$, we have
\begin{align*}
    \angle (x_0 - x_i^*, \; x_0 - \hat{x}) 
    &< \arcsin \Big( \frac{ \Tilde{R}_i }{ \Tilde{R}_i \sec \theta_i } \Big) \\
    &=\arcsin ( \cos \theta_i )  \\
    &= \frac{\pi}{2} - \theta_i.
\end{align*}
Using Proposition \ref{prop: grad} and the inequality above, we can bound the angle between the vectors $ g_i(x_0)$ and $x_0 - \hat{x}$ as follows:
\begin{align*}
    & \quad \angle ( g_i(x_0), \; x_0 - \hat{x}) \\
    &\leq \angle ( g_i(x_0), \; x_0 - x_i^*) + \angle (x_0 - x_i^*, \; x_0 - \hat{x})  \\
    &< \theta_i + \Big( \frac{\pi}{2} - \theta_i \Big) = \frac{\pi}{2}.
\end{align*}
Note that the first inequality is obtained from \cite[Corollary~12]{castano2016angles}. Let $u = \frac{x_0 - \hat{x}}{ \| x_0 - \hat{x} \|_2 }$. Compute the inner product
\begin{align*}
    \Big\langle \sum_{v_i \in \mathcal{R}} g_i(x_0), \; u \Big\rangle 
    &= \sum_{v_i \in \mathcal{R}} \langle g_i(x_0), \; u \rangle \\
    &= \sum_{v_i \in \mathcal{R}} \| g_i(x_0) \|_2 \cos \angle ( g_i(x_0), \; x_0 - \hat{x}) \\
    &> 0
\end{align*}
since $\| g_i(x_0) \|_2 > 0$ and $\cos \angle ( g_i(x_0), \; x_0 - \hat{x}) > 0$ for any $v_i \in \mathcal{R}$. This implies that $\sum_{v_i \in \mathcal{R}} g_i(x_0) \neq \mathbf{0}$. Since we can arbitrarily choose $g_i(x_0)$ from the set $\partial f_i(x_0)$, we have $\mathbf{0} \notin \partial f(x_0)$ where $f(x) = \frac{1}{| \mathcal{R} |} \sum_{v_i \in \mathcal{R}} f_i(x)$. 
\end{proof}

%Note that Theorem~\ref{thm: main} holds for any filtering algorithms that satisfy the general property stated in Proposition~\ref{prop: alg prop}. 

Thus, Theorem~\ref{thm: main} and Proposition~\ref{prop: true_sol} show that Algorithm~\ref{alg: resilient dec opt} causes all regular nodes to converge to a region that also contains the true solution, regardless of the actions of any $F$-local set of Byzantine adversaries.  The size of this region scales with the quantity $R^*$. Loosely speaking, this quantity becomes smaller as the minimizers of the local functions of the regular agents get closer together.  More specifically, consider a fixed $\epsilon > 0$.  If the functions $f_i(x)$ are translated so that the minimizers $x_i^*$ get closer together (i.e., $\Tilde{R}_i$ is smaller and $\theta(\epsilon)$ is fixed), and the auxiliary point $\hat{x}$ is in the hyperrectangle containing the minimizers (which is the case when we run a resilient consensus algorithm such as the one in \cite{leblanc2013resilient}) then $R^*$ also decreases, and the state $x_i[k]$ is guaranteed to become closer to the true minimizer as $k$ goes to infinity.

\section{Numerical Experiment}

\begin{figure}[h]
\centering
{\includegraphics[width=0.45\textwidth]{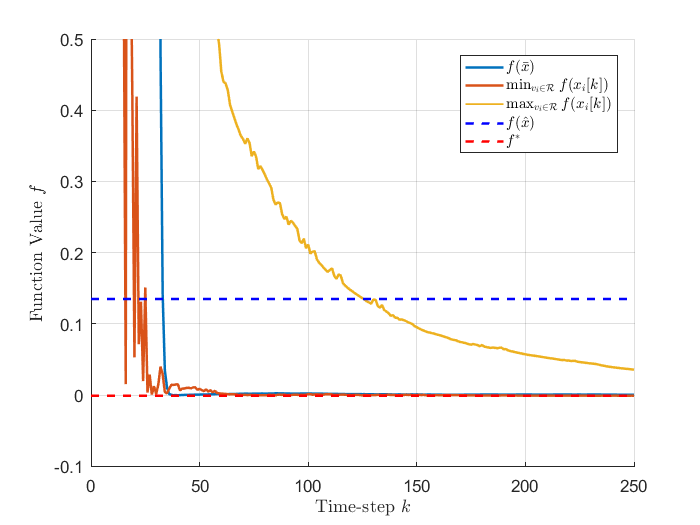}}
\caption{Function value $f(x)$ evaluated at the average of regular agents' states (blue solid line), at the regular agent's state that gives the lowest and highest value (orange and yellow solid line respectively), at the auxiliary point (blue dashed line), and at the minimizer (red dashed line)}
\label{fig: numerical experiment}
\end{figure}

In our numerical experiment, we set the total number of agents to be $n = 100$ and the number of dimensions to be $d = 3$. We construct a $15$-robust network in order to tolerate up to $2$  Byzantine agents in the neighborhood of any regular nodes i.e., $F=2$. Each regular node possesses a quadratic function $f_i(x) = \frac{1}{2} x^T Q_i x + b_i^T x$ where $Q_i$ and $b_i$ are randomly chosen with $Q_i$ guaranteed to be positive definite.\footnote{To satisfy Assumption 2 (bounded gradients), we saturate the gradients during the updates when the norms are sufficiently large.} 
Each Byzantine agent selects the transmitted vector based on the target regular node. Specifically, the transmitted vector is picked  uniformly at random inside the hyper-rectangle that guarantees that its state is not  discarded by the target node. %Note that each Byzantine agent need an oracle to observe the set of states that the target possesses to make the decision.

Let $f (x) \triangleq \frac{1}{| \mathcal{R} |} \sum_{v_i \in \mathcal{R}} f_i(x)$ be the objective function (Problem \ref{prob: regular node}) evaluated at $x$, $f^* \triangleq \min_{x \in \mathbb{R}^d} f(x)$ be the optimal value of the objective function and $\bar{x}[k] \triangleq \frac{1}{| \mathcal{R} |} \sum_{v_i \in \mathcal{R}} x_i[k]$ be the average of the regular nodes' states at time-step $k$.
In Figure \ref{fig: numerical experiment}, the objective value evaluated at the average of all regular states $f ( \bar{x}[k] )$ is near the optimal value $f^*$ after $40$ iterations. Furthermore, the maximum of the objective values among all regular agents $\max_{v_i \in \mathcal{R}} f( x_i[k] )$ and the minimum of the objective values among all regular agents $\min_{v_i \in \mathcal{R}} f( x_i[k] )$ converge to the same value, which is an implication of the fact that the regular agents reach consensus (from Theorem~\ref{thm:consensus}).   

\begin{remark}
The bound in Theorem \ref{thm: main}, for large $k$, does not preclude %even though $\bar{x}[k] \in \mathcal{B}( \hat{x}, R^* + r)$ where $r$ is a small constant, we might have 
$\| \bar{x}[k] - x^* \|_2 > \| \hat{x} - x^* \|_2$, where $x^*$ is a minimizer of $f$, and $f ( \bar{x}[k] ) > f( \hat{x} )$.
However, our experiments (with various settings), including the one above, show that the objective function value evaluated at the average of the states is much closer to the optimal value than that of the auxiliary point in practice, i.e., $f( \bar{x}[k] ) - f^* \ll f( \hat{x} ) - f^*$.
\end{remark}

\section{Conclusion and Future work}\label{sec:Conclusion and Future work}

In this paper, we developed a resilient distributed optimization algorithm for multi-dimensional functions. Our results  guarantee that the regular states asymptotically reach consensus and enter a bounded region that contains the global minimizer, irrespective of the actions of Byzantine agents, if the network topology satisfies certain conditions. We characterized the size of this region.  A promising avenue for future research would be to further refine the size of the convergence region, and to relax the conditions on the network topology. 

\appendix

%\section{Proof of the Main Results}\label{sec:Proof}

%\subsection{Proof of Proposition \ref{prop: grad}}

%Now, we will show that using Assumption \ref{ass: gradient_bound} and Assumption  \ref{ass: convex} implies that the angle $\angle ( - g(x), \; x^* - x )$ is bounded by a constant which is strictly less than $\frac{\pi}{2}$.

\subsection{Additional Lemma}
We provide a lemma which will be utilized in \textbf{Step 2} of the proof of Theorem~\ref{thm: main}.
\begin{lemma}  \label{lem: max_angle}
If $x \notin \mathcal{B}(\hat{x}, R)$ then
\begin{align*}
    \max_{y \in \mathcal{B}(\hat{x}, R)}  \angle (x - y, \; x - \hat{x})  = \arcsin \Big( \frac{ R }{ \| x - \hat{x} \|_2} \Big). 
\end{align*}
\end{lemma} 

\begin{proof}
Since the angle is measured with respect to the vector $x - \hat{x}$, consider any 2-D planes passed through the center $\hat{x}$ and the point $x$. Since the planes pass through $\hat{x}$, the intersections between of the ball $\mathcal{B}(\hat{x}, R)$ and the planes are great circles of radius $R$. Thus, all of the intersections generated from each plane are identical and we can consider the angle using a great circle instead of the ball. From geometry, the maximum angle $\phi = \angle (x - y^*, \; x - \hat{x})$ only occurs when the ray starting from the point $x$ touches the circle at point $y^*$. Therefore, $\angle ( \hat{x}-y^*, x-y^*) = \frac{\pi}{2}$ and $\| \hat{x} - y^* \|_2 = R$. We have
\begin{align*}
    \sin \phi = \frac{ \| \hat{x} - y^* \|_2 }{ \| \hat{x} - x \|_2 } = \frac{R}{\| \hat{x} - x \|_2}
\end{align*}
and the result follows.
\end{proof}

\subsection{Proof of Theorem~\ref{thm: main}}
We work towards the proof of Theorem \ref{thm: main} in several steps.  
For a fixed $\epsilon > 0$ and for all $\xi \in \mathbb{R}_{> 0}$, define the \textit{convergence radius}
\begin{align*}
    s^*( \xi ) \triangleq \max_{v_i \in \mathcal{R}} \big\{ \max \{ \Tilde{R}_i \sec \theta_i, \Tilde{R}_i + \delta_i \} \big\} + \xi.
\end{align*}
The key idea is that we partition the space $\mathbb{R}^d$ into $3$ regions: 
\begin{enumerate}[1.]
    \item $\mathcal{B} (\hat{x}, \; \max_{v_i \in \mathcal{R}} \{ \Tilde{R}_i + \delta_i \} )$,   \label{rgn2: 1}
    \item $\mathcal{B} (\hat{x}, \; s^*(\xi) ) \setminus \mathcal{B} (\hat{x}, \; \max_{v_i \in \mathcal{R}} \{ \Tilde{R}_i + \delta_i \} )$, and   \label{rgn2: 2}
    \item $\mathbb{R}^d \setminus \mathcal{B} (\hat{x}, \; s^*(\xi) )$.  \label{rgn2: 3}
\end{enumerate}

We refer to $\mathcal{B} (\hat{x}, \; s^*(\xi) )$ as the \textit{convergence region}.  First, for sufficiently large $k$, when every state $x_i[k]$ is in the convergence region (region 1 or 2), by using Proposition~\ref{prop: alg prop} and the gradient update \eqref{eqn: dynamic}, we show that the regular states cannot leave the convergence region (Steps 1 to 3 below).  Next, we show that if some state $x_i[k]$ is outside of the convergence region (i.e., in region \ref{rgn2: 3}), then by applying Proposition~\ref{prop: alg prop} and the gradient update \eqref{eqn: dynamic} repeatedly, every state will eventually enter the convergence region (Step 4). These steps will lead to the proof of Theorem~\ref{thm: main}.  Throughout the following exposition, we fix $\xi$ to be a strictly positive constant.

\textbf{Step 1:} We want to show that if $z_i[k] \in \mathcal{B} (\hat{x}, \max_{v_j \in \mathcal{R}} \{ \Tilde{R}_j + \delta_j \} ) \subset \mathcal{B} (\hat{x}, s^*( \xi ) )$ and $k$ is large enough, then by applying the gradient update \eqref{eqn: dynamic}, the state $x_i[k+1] \in \mathcal{B} (\hat{x}, s^*( \xi ) )$ i.e., $x_i[k+1]$ is still in the convergence region.  Let $k^*_1 \in \mathbb{N}$ be such that $\eta[k^*_1] \leq \frac{\xi}{L}$.
\begin{lemma} \label{lem: leq R_i}
For all $v_i \in \mathcal{R}$ and $k \geq k^*_1$, if $\| z_i[k] - \hat{x} \|_2 \leq \max_{v_j \in \mathcal{R}} \{ \Tilde{R}_j + \delta_j \} $ then $\| x_i[k+1] - \hat{x} \|_2 \leq s^*( \xi )$.  
\end{lemma}
\begin{proof}
From the update \eqref{eqn: dynamic}, we have
\begin{align*}
    \| x_i[k+1] - \hat{x} \|_2 &= \| z_i[k] - \hat{x} - \eta[k] g_i[k] \|_2 \\
    &\leq \| z_i[k] - \hat{x} \|_2 + \eta[k] \| g_i[k] \|_2  .
\end{align*}
Using Assumption \ref{ass: gradient_bound} and \ref{ass: step-size}, and $\| z_i[k] - \hat{x} \|_2 \leq \max_{v_j \in \mathcal{R}} \{ \Tilde{R}_j + \delta_j \} $, we obtain
\begin{align*}
    \| x_i[k+1] - \hat{x} \|_2 &\leq \max_{v_j \in \mathcal{R}} \{ \Tilde{R}_j + \delta_j \}  + \eta[k] L  \\
    &\leq \max_{v_j \in \mathcal{R}} \{ \Tilde{R}_j + \delta_j \}  + \eta[k_1^*] L  .
\end{align*}
By the definition of $k_1^*$ and $s^*( \xi )$, the above inequality becomes
\begin{align*}
    \| x_i[k+1] - \hat{x} \|_2 \leq \max_{v_j \in \mathcal{R}} \{ \Tilde{R}_j + \delta_j \}  + \xi 
    \leq s^*( \xi ). 
\end{align*}
\end{proof}

\textbf{Step 2:} We want to analyze the relationship between the terms $\| x_i[k+1] - \hat{x} \|_2$ and $\| z_i[k] - \hat{x} \|_2$ which will be used in the subsequent step.

For $v_i \in \mathcal{R}$, define $\Delta_i: [ \Tilde{R}_i, \infty) \times \mathbb{R}_+ \to \mathbb{R}$ to be the function
\begin{align}
    \Delta_i( p, l ) \triangleq 2 l  (\sqrt{ p^2 - \Tilde{R}_i^2 } \cos \theta_i - \Tilde{R}_i \sin \theta_i) - l^2,  \label{def: delta}
\end{align}
and $\Gamma_i: [ \Tilde{R}_i, \infty) \times \mathbb{R}_+ \to \mathbb{R}$ to be the function
\begin{align}
    \Gamma_i( p, l ) \triangleq p^2 - \Delta_i( p, l ) .   \label{def: gamma}
\end{align}
\begin{lemma}
For all $v_i \in \mathcal{R}$ and $k \in \mathbb{N}$, if $\| z_i[k] - \hat{x} \|_2 > \Tilde{R}_i + \delta_i$ then 
\begin{align*}
    &\| x_i[k+1] - \hat{x} \|_2^2 \\
    &\leq \| z_i[k] - \hat{x} \|_2^2 - \Delta_i( \| z_i[k] - \hat{x} \|_2, \; \eta[k] \; \| g_i[k] \|_2 )  \\
    &= \Gamma_i( \| z_i[k] - \hat{x} \|_2, \; \eta[k] \; \| g_i[k] \|_2 ).  
\end{align*}  \label{lem: z_to_x}
\end{lemma}
\begin{proof}
Consider $v_i \in \mathcal{R}$. Define $\| z_i[k] - \hat{x} \|_2 \triangleq s_i[k]$ and $z_i[k] \notin \mathcal{B}(\hat{x}, \Tilde{R}_i + \delta_i)$. By Lemma~\ref{lem: max_angle}, we have
\begin{align}
    & \qquad \angle (x_i^* - z_i[k], \hat{x} - z_i[k])  \nonumber \\
    &\leq \max_{y \in \mathcal{B}(\hat{x}, \Tilde{R}_i)}  \angle (y - z_i[k], \hat{x} - z_i[k])  \nonumber \\
    &= \arcsin \frac{\Tilde{R}_i}{s_i[k]} \triangleq \phi_i(s_i[k]). \label{def: angle_phi}
\end{align}
Let $\psi_i[k] \triangleq \angle ( x_i[k+1] - z_i[k], \; \hat{x} - z_i[k]) = \angle ( - \eta[k] g_i[k], \; \hat{x} - z_i[k] )$.
Since for all $v_i \in \mathcal{R}$ and $k$,
\begin{multline*}
    \angle ( - \eta[k] g_i[k], \; \hat{x} - z_i[k] ) \leq \angle ( - \eta[k] g_i[k], \; x_i^* - z_i[k]  ) \\
    + \angle ( x_i^* - z_i[k], \; \hat{x} - z_i[k])
\end{multline*}
by \cite[Corollary~12]{castano2016angles}, applying Proposition \ref{prop: grad} and inequality \eqref{def: angle_phi}, we have
\begin{align}
    \psi_i[k] \leq \theta_i + \phi_i(s_i[k]) \triangleq \Tilde{\psi}_i(s_i[k]).  \label{eqn: angle}
\end{align}
Note that $\Tilde{\psi}_i(s_i[k]) < \pi$ since $\theta_i < \frac{\pi}{2}$ and $\phi_i(s_i[k]) \leq \frac{\pi}{2}$.
Consider the triangle which has the vertices at $x_i[k+1]$, $z_i[k]$, and $\hat{x}$.
We can calculate the square of the distance by using the law of cosines:
\begin{multline*}
    \| x_i[k+1] - \hat{x} \|_2^2 = \; \| x_i[k+1] - z_i[k] \|^2_2 + \| \hat{x} - z_i[k] \|^2_2 \\
    - 2 \| x_i[k+1] - z_i[k] \|_2 \cdot \| z_i[k] - \hat{x} \|_2 \\
    \times \cos \angle (x_i[k+1] - z_i[k], \; \hat{x} - z_i[k]  )  .
\end{multline*}    
Using the update \eqref{eqn: dynamic} and $\| z_i[k] - \hat{x} \|_2 = s_i[k]$, we get 
\begin{align}    
    \| x_i[k+1] - \hat{x} \|_2^2 &=   \eta[k]^2 \; \| g_i[k] \|_2^2 + s_i[k]^2 \nonumber \\
    & \quad - 2 s_i[k] \; \eta[k] \; \| g_i[k] \|_2 \cos \psi_i[k]  \nonumber \\
    &\leq  \eta[k]^2 \; \| g_i[k] \|_2^2 + s_i[k]^2 \nonumber \\
    &\quad - 2 s_i[k] \; \eta[k] \; \| g_i[k] \|_2 \cos \Tilde{\psi}_i(s_i[k]). \label{eqn: dist1}
\end{align}
The inequality above is obtained by using the inequality \eqref{eqn: angle}. In addition, we can simplify the term $s_i[k] \cos \Tilde{\psi}_i(s_i[k])$ as follows:
\begin{align*}
    s_i[k] \cos \Tilde{\psi}_i(s_i[k]) &= s_i[k] \cos \Big( \arcsin \frac{\Tilde{R}_i}{s_i[k]} + \theta_i \Big)  \\
    &= s_i[k] \cos \Big( \arcsin \frac{\Tilde{R}_i}{s_i[k]} \Big) \cos \theta_i \\
    &\quad - s_i[k] \sin \Big( \arcsin \frac{\Tilde{R}_i}{s_i[k]} \Big) \sin \theta_i  \\
    &= \sqrt{ s_i[k]^2 - \Tilde{R}_i^2 } \cos \theta_i - \Tilde{R}_i \sin \theta_i . 
\end{align*}
Therefore, we can rewrite \eqref{eqn: dist1} as 
\begin{multline*}
    \| x_i[k+1] - \hat{x} \|_2^2 \leq \eta[k]^2 \; \| g_i[k] \|_2^2 + s_i[k]^2 \\
    - 2 \; \eta[k] \; \| g_i[k] \|_2 \; (\sqrt{ s_i[k]^2 - \Tilde{R}_i^2 } \cos \theta_i - \Tilde{R}_i \sin \theta_i). 
\end{multline*}
\end{proof}

\textbf{Step 3:} We want to show that if $z_i[k] \in \mathcal{B} (\hat{x}, s^*( \xi ) )$, $z_i[k] \notin \mathcal{B} (\hat{x}, \max_{v_j \in \mathcal{R}} \{ \Tilde{R}_j + \delta_j \} )$ and $k$ is large enough, then by applying the gradient update \eqref{eqn: dynamic}, the state $x_i[k+1] \in \mathcal{B} (\hat{x}, s^*( \xi ) )$ i.e., $x_i[k+1]$ is still in the convergence region.

Let 
\begin{align*}
    a_i^+ &\triangleq -\Tilde{R}_i \sin \theta_i + \sqrt{ \big( s^*(\xi) \big)^2 - \Tilde{R}_i^2 \cos^2 \theta_i } \; , \\
    a_i^- &\triangleq -\Tilde{R}_i \sin \theta_i - \sqrt{ \big( s^*(\xi) \big)^2 - \Tilde{R}_i^2 \cos^2 \theta_i } \;, \quad \text{and} \\
    b_i &\triangleq 2 \big( \sqrt{ \big( s^*(\xi) \big)^2 - \Tilde{R}_i^2 } \cos \theta_i - \Tilde{R}_i \sin \theta_i \big).
\end{align*}
\noindent Let $k_2^*$ be such that 
\begin{align*}
    \eta[k_2^*] \leq \frac{1}{L} \min_{v_i \in \mathcal{R}} \big\{ \min \{ a_i^+, \; b_i \} \big\}.
\end{align*}
\begin{lemma}
For all $i \in \mathcal{R}$ and $k \geq k_2^*$, if $\| z_i[k] - \hat{x} \|_2 \in \big( \max_{v_j \in \mathcal{R}} \{ \Tilde{R}_j + \delta_j \}, s^*(\xi) \big]$ then $\| x_i[k+1] - \hat{x} \|_2   \leq s^*(\xi)$.   \label{lem: leq R_i to s_i}
\end{lemma}
\begin{proof}
Consider the function $\Gamma_i( p, l )$ defined in \eqref{def: gamma}. Compute the second derivative with respect to $p$:
\begin{align*}
    &\frac{\partial \Gamma_i}{\partial p} = 2p - 2 l p  (p^2 - \Tilde{R}_i^2)^{- \frac{1}{2}} \cos \theta_i, \\
    &\frac{\partial^2 \Gamma_i}{\partial p^2} = 2 + 2 l \Tilde{R}_i^2 (p^2 - \Tilde{R}_i^2)^{- \frac{3}{2}} \cos \theta_i .
\end{align*}
Note that $\frac{\partial^2 \Gamma_i}{\partial p^2} > 0$ for all $p \in (\Tilde{R}_i, \infty)$. This implies that 
\begin{multline}
    \sup_{ p \in ( \max_{v_j \in \mathcal{R}} \{ \Tilde{R}_j + \delta_j \}, \; s^*(\xi) ] } \Gamma_i (p, l) 
    \leq \max_{ p \in [ \Tilde{R}_i, \; s^*(\xi) ] } \Gamma_i (p, l)   \\
     = \max \big\{ \Gamma_i( \Tilde{R}_i, l ), \; \Gamma_i ( s^*(\xi), l ) \big\}.  \label{eqn: gamma chain}
\end{multline} 
First, let consider $\Gamma_i( \Tilde{R}_i, \; l )$.
From the definition of $\Gamma_i$ in \eqref{def: gamma}, we have that
\begin{align}
    \Gamma_i ( \Tilde{R}_i, l ) \leq \big( s^*(\xi) \big)^2 \quad \Longleftrightarrow \quad l \in [ a_i^-, \; a_i^+ ].   \label{eqn: gamma_leq_1}
\end{align}
Note that $a_i^+ > 0$ and $a_i^- < 0$ since $ s^*(\xi)  \geq \Tilde{R}_i + \xi$.
Using Assumption \ref{ass: gradient_bound}, Assumption \ref{ass: step-size}, and the definition of $k_2^*$, we have that for all $v_i \in \mathcal{R}$ and $k \geq k_2^*$,
\begin{align*}
    \eta[k] \; \| g_i[k] \|_2 \leq \eta[k] L \leq \min_{v_j \in \mathcal{R}} \big\{ \min \{ a_j^+, \; b_j \} \big\} \leq  a_i^+ .
\end{align*}
By \eqref{eqn: gamma_leq_1}, we obtain that for all $v_i \in \mathcal{R}$ and $k \geq k_2^*$,
\begin{align}
     \Gamma_i ( \Tilde{R}_i, \; \eta[k] \; \| g_i[k] \|_2 ) \leq \big( s^*(\xi) \big)^2. \label{eqn: gamma ineq1}
\end{align}
Now, let consider $\Gamma_i ( s^*(\xi), \; l )$. 
From the definition of $\Gamma_i$ in \eqref{def: gamma}, we have that
\begin{align}
     \Gamma_i ( s^*(\xi), l ) \leq \big( s^*(\xi) \big)^2  \quad \Longleftrightarrow \quad
     l \in [ 0, \; b_i ].  \label{eqn: gamma_leq_2}
\end{align}
Note that $b_i > 0$ since $s^*(\xi)  \geq \Tilde{R}_i \sec \theta_i + \xi$.
Using Assumption \ref{ass: gradient_bound}, Assumption \ref{ass: step-size}, and the definition of $k_2^*$, we have that for all $v_i \in \mathcal{R}$ and $k \geq k_2^*$,
\begin{align*}
    \eta[k] \; \| g_i[k] \|_2 \leq \eta[k] L \leq \min_{v_j \in \mathcal{R}} \big\{ \min \{ a_j^+, \; b_j \} \big\} \leq  b_i .
\end{align*}
By \eqref{eqn: gamma_leq_2}, we obtain that for all $v_i \in \mathcal{R}$ and $k \geq k_2^*$,
\begin{align}
     \Gamma_i ( s^*(\xi), \; \eta[k] \; \| g_i[k] \|_2 ) \leq \big( s^*(\xi) \big)^2.  \label{eqn: gamma ineq2}
\end{align}
Combine \eqref{eqn: gamma ineq1} and \eqref{eqn: gamma ineq2} to get that for all $v_i \in \mathcal{R}$ and $k \geq k_2^*$,
\begin{multline}
    \max \big\{ \Gamma_i ( \Tilde{R}_i, \; \eta[k] \; \| g_i[k] \|_2 ), \; \Gamma_i ( s^*(\xi), \; \eta[k] \; \| g_i[k] \|_2 ) \big\} \\
    \leq \big( s^*(\xi) \big)^2. \label{eqn: max bound s}
\end{multline}
Using Lemma \ref{lem: z_to_x}, \eqref{eqn: gamma chain} and \eqref{eqn: max bound s}, respectively, we obtain that for all $v_i \in \mathcal{R}$ and $k \geq k_2^*$,
\begin{align*}
    & \quad \| x_i[k+1] - \hat{x} \|_2^2  \\
    &\leq \Gamma_i ( \| z_i[k] - \hat{x} \|_2, \; \eta[k] \; \| g_i[k] \|_2 )  \\
    &\leq \sup_{ p \in ( \max_{v_j \in \mathcal{R}} \{ \Tilde{R}_j + \delta_j \}, \; s^*(\xi) ] } \Gamma_i (p, \; \eta[k] \; \| g_i[k] \|_2) \\
    &\leq \max \big\{ \Gamma_i ( \Tilde{R}_i, \; \eta[k] \; \| g_i[k] \|_2 ), \;
    \Gamma_i ( s^*(\xi), \; \eta[k] \; \| g_i[k] \|_2 ) \big\} \\
    &\leq \big( s^*(\xi) \big)^2.
\end{align*}
\end{proof}

\textbf{Step 4:} %Intuitively, when the distance $\| x_i[k] - \hat{x} \|_2 >  s^*(\xi)$, by applying the Proposition~\ref{prop: alg prop} and the gradient update \eqref{eqn: dynamic}, respectively, the distance $\| x_i[k] - \hat{x} \|_2$ decreases.
We now show that the states of all regular nodes eventually enter the convergence region.

\begin{lemma}
    There exists $k \geq \max \{ k_1^*, k_2^*  \}$ such that for all $v_i \in \mathcal{R}$, $ \| x_i[k] - \hat{x} \|_2 \leq  s^*(\xi) $.  \label{lem: geq s}
\end{lemma} 
\begin{proof}
Assume for the purpose of contradiction that for all $k \geq \max \{ k_1^*, k_2^*  \}$ there exists $v_i \in \mathcal{R}$ such that $\| x_i[k] - \hat{x} \|_2 >  s^*(\xi)$.
Define the sets of nodes
\begin{align*}
    \mathcal{I}_x[k] \triangleq \{ v_i \in \mathcal{R} : \| x_i[k] - \hat{x} \|_2 >  s^*(\xi) \}
\end{align*}
and
\begin{align*}
    \mathcal{I}_z[k] \triangleq \{ v_i \in \mathcal{R} : \| z_i[k] - \hat{x} \|_2 >  s^*(\xi) \}.
\end{align*}
Note that $\mathcal{I}_x[k] \neq \emptyset$ for all $k \geq \max \{ k_1^*, k_2^* \}$ by the assumption. In addition, from Lemma \ref{lem: leq R_i} and Lemma \ref{lem: leq R_i to s_i}, we have that for all $v_i \in \mathcal{R}$ and $k \geq \max \{ k_1^*, k_2^* \}$, if $\| z_i[k] - \hat{x} \|_2 \leq s^*(\xi)$ then $\| x_i[k+1] - \hat{x} \|_2   \leq s^*(\xi)$. This implies that for all $k \geq \max \{ k_1^*, k_2^* \}$, $\emptyset \neq \mathcal{I}_x [k+1] \subseteq \mathcal{I}_z[k]$.

From Lemma \ref{lem: z_to_x}, for $k \geq \max \{ k_1^*, k_2^* \}$ and $v_i \in \mathcal{I}_z[k]$,
\begin{multline*}
    \| x_i[k+1] - \hat{x} \|_2^2 \leq \| z_i[k] - \hat{x} \|_2^2 \\
    - \Delta_i( \| z_i[k] - \hat{x} \|_2, \; \eta[k] \; \| g_i[k] \|_2 ).  
\end{multline*}
For any $k \geq \max \{ k_1^*, k_2^*  \}$, we have
\begin{align}
    &\quad \max_{v_i \in \mathcal{I}_z[k]} \| x_i[k+1] - \hat{x} \|_2^2 \nonumber \\
    &\leq \max_{v_i \in \mathcal{I}_z[k]} \Big( \| z_i[k] - \hat{x} \|_2^2 \nonumber \\
    & \qquad \qquad - \Delta_i( \| z_i[k] - \hat{x} \|_2,  \eta[k] \; \| g_i[k] \|_2 ) \Big)  \nonumber \\
    &\leq \max_{v_i \in \mathcal{I}_z[k]}  \| z_i[k] - \hat{x} \|_2^2  \nonumber \\
    & \qquad \qquad - \min_{v_i \in \mathcal{I}_z[k]} \Delta_i( \| z_i[k] - \hat{x} \|_2, \; \eta[k] \; \| g_i[k] \|_2 ) .  \label{eqn: max-min}
\end{align}
Consider the first term on the RHS of \eqref{eqn: max-min}. By the definition of $\mathcal{I}_z[k]$, we have
\begin{align*}
    \max_{v_i \in \mathcal{I}_z[k]}  \| z_i[k] - \hat{x} \|_2^2 = \max_{v_i \in \mathcal{R}}  \| z_i[k] - \hat{x} \|_2^2.
\end{align*}
Consider the term on the LHS of \eqref{eqn: max-min}. Since $\emptyset \neq \mathcal{I}_x [k+1] \subseteq \mathcal{I}_z[k]$ and the definition of $\mathcal{I}_x[k+1]$, we have
\begin{align*}
    \max_{v_i \in \mathcal{I}_z[k]} \| x_i[k+1] - \hat{x} \|_2^2 
    &\geq  \max_{v_i \in \mathcal{I}_x[k+1]} \| x_i[k+1] - \hat{x} \|_2^2 \\
    &= \max_{v_i \in \mathcal{R}} \| x_i[k+1] - \hat{x} \|_2^2.
\end{align*}
Therefore, the inequality \eqref{eqn: max-min} becomes
\begin{multline}
    \max_{v_i \in \mathcal{R}} \| x_i[k+1] - \hat{x} \|_2^2 \leq \max_{v_i \in \mathcal{R}}  \| z_i[k] - \hat{x} \|_2^2  \\
    - \min_{v_i \in \mathcal{I}_z[k]} \Delta_i( \| z_i[k] - \hat{x} \|_2, \; \eta[k] \; \| g_i[k] \|_2 ) .  \label{eqn: max-min2}
\end{multline}

Consider the definition of the function $\Delta_i$ in \eqref{def: delta}. It is clear that if $p_1 > p_2 \geq \Tilde{R}_i$ then $\Delta_i( p_1, l) > \Delta_i( p_2, l)$.
Then, for all $k \geq \max \{ k_1^*, k_2^*  \}$ and $v_i \in \mathcal{I}_z[k]$, we get 
\begin{align}
    \Delta_i( \| z_i[k] - \hat{x} \|_2, \eta[k] \; \| g_i[k] \|_2 ) > \Delta_i (  s^*(\xi), \eta[k] \; \| g_i[k] \|_2 ).  \label{eqn: delta_ineq}
\end{align}
Furthermore, the function $\Delta_i$ satisfies
\begin{multline}
    \Delta_i ( p, l) \geq \Big( \sqrt{ p^2 - \Tilde{R}_i^2 } \cos \theta_i - \Tilde{R}_i \sin \theta_i \Big) l \\
    \Longleftrightarrow \quad l \in \Big[ 0, \; \sqrt{ p^2 - \Tilde{R}_i^2 } \cos \theta_i - \Tilde{R}_i \sin \theta_i  \Big].  \label{eqn: l_range}
\end{multline}

For $x \notin \mathcal{C}_i(\epsilon)$, from the definition of convex functions, we have
$- \langle g_i(x), \; x_i^* - x \rangle \geq f_i(x) - f_i(x_i^*)$.  
Using the inequality \eqref{eqn: lb_grad}, we obtain
\begin{align}
    \| g_i (x) \|_2 \geq
    \frac{f_i(x) - f_i(x_i^*)}{\| x - x_i^* \|_2}
    \geq \frac{\epsilon}{\delta_i(\epsilon)} \triangleq \kappa_i (\epsilon) > 0, \label{eqn: grad_lb}  
\end{align}
where $\mathcal{C}_i(\epsilon)$ is defined in \eqref{def: sublevel_set}. Let $k_3^*$ be such that 
\begin{align*}
    \eta[k_3^*] \leq \frac{1}{2L} \min_{v_i \in \mathcal{R}} b_i.
\end{align*}
Using \eqref{eqn: grad_lb}, Assumption \ref{ass: gradient_bound}, and the definition of $k_3^*$, we have that for all $k \geq k_3^*$ and $v_i \in \mathcal{I}_z[k]$, 
\begin{align}
    \eta[k] \kappa_i \leq \eta[k] \; \| g_i[k] \|_2 \leq \eta[k] L \leq \frac{b_i}{2}.  \label{eqn: lb_ub}
\end{align}
Since $\eta[k] \; \| g_i[k] \|_2 \in [ 0, \; \frac{b_i}{2}  ]$, we can apply \eqref{eqn: l_range} to get
\begin{align}
    \Delta_i \big(  s^*(\xi), \; \eta[k] \; \| g_i[k] \|_2 \big) \geq \frac{b_i}{2} \eta[k] \; \| g_i[k] \|_2.  \label{eqn: delta_geq}
\end{align}
Let $k^* = \max \{  k_1^*, k_2^*, k_3^* \}$.
Combine \eqref{eqn: delta_ineq}, \eqref{eqn: delta_geq}, and the first inequality of \eqref{eqn: lb_ub} to get that for all $k \geq k^*$ and $v_i \in \mathcal{I}_z[k]$, 
\begin{align}
    \Delta_i( \| z_i[k] - \hat{x} \|_2, \eta[k] \; \| g_i[k] \|_2 ) > \frac{1}{2} b_i \kappa_i \eta[k]  .  \label{eqn: grad_chain2}
\end{align}
On the other hand, from Proposition~\ref{prop: alg prop}, we have that for all $v_i \in \mathcal{R}$,
\begin{align*}
    \| z_i[k] - \hat{x} \|_2 
    &\leq \max_{v_j \in (\mathcal{N}_i^- \cap \mathcal{R}) \cup \{ v_i \} } \| x_j[k] - \hat{x} \|_2 \\
    &\leq \max_{v_j \in \mathcal{R}} \| x_j[k] - \hat{x} \|_2 ,
\end{align*}
which implies that 
\begin{align*}
     \max_{v_i \in \mathcal{R}} \| z_i[k] - \hat{x} \|_2 \leq \max_{v_i \in \mathcal{R}} \| x_i[k] - \hat{x} \|_2.
\end{align*}
Apply the above inequality and the inequality \eqref{eqn: grad_chain2} to \eqref{eqn: max-min2} to get that for all $k \geq k^*$,
\begin{align}
    & \quad \max_{v_i \in \mathcal{R}} \| x_i[k+1] - \hat{x} \|_2^2 \nonumber \\
    &< \max_{v_i \in \mathcal{R}}  \| x_i[k] - \hat{x} \|_2^2  
    - \frac{1}{2} \min_{v_i \in \mathcal{I}_z[k]} b_i \kappa_i  \eta[k] \nonumber \\
    &\leq \max_{v_i \in \mathcal{R}}  \| x_i[k] - \hat{x} \|_2^2  
    - \frac{1}{2} \min_{v_i \in \mathcal{R}} b_i \kappa_i  \eta[k].
    \label{eqn: max-min4}
\end{align}
Apply the inequality \eqref{eqn: max-min4} recursively to get that for $K \geq 1$,
\begin{multline*}
    \max_{v_i \in \mathcal{R}} \| x_i[k^*+K] - \hat{x} \|_2^2 < \max_{v_i \in \mathcal{R}}  \| x_i[k^*] - \hat{x} \|_2^2 \\
    - \frac{1}{2} \min_{v_i \in \mathcal{R}} b_i \kappa_i   \sum_{k=k^*}^{k^* + K - 1} \eta[k] . 
\end{multline*}
Since $\sum_{k=1}^{\infty} \eta [k] = \infty$ from Assumption \ref{ass: step-size}, we have that there exists $K^*$ such that 
\begin{align*}
    \max_{ v_i \in \mathcal{R}} \| x_i[k^* + K^*] - \hat{x} \|_2^2 \leq \big( s^*(\xi) \big)^2.
\end{align*}
This implies that for all $v_i \in \mathcal{R}$,
\begin{align*}
    \| x_i[k^* + K^*] - \hat{x} \|_2^2 \leq \big( s^*(\xi) \big)^2
\end{align*}
which contradicts our assumption.
\end{proof}

We are now in place to prove Theorem \ref{thm: main} by combining the results from Lemma \ref{lem: leq R_i}, \ref{lem: leq R_i to s_i}, and \ref{lem: geq s}. Intuitively, Lemma \ref{lem: geq s} states that there is a time-step $k$ that $x_i[k] \in \mathcal{B}( \hat{x}, s^*(\xi) )$ for all $v_i \in \mathcal{R}$. Then, by applying Lemma \ref{lem: leq R_i} and \ref{lem: leq R_i to s_i}, and Proposition~\ref{prop: alg prop}, we can conclude that the state $x_i[k]$ will remain inside $\mathcal{B}( \hat{x}, s^*(\xi) )$ for all subsequent time-steps.  

\begin{proof}[Proof of Theorem \ref{thm: main}]
Let $k^* \geq \max \big\{ k_1^*, k_2^* \big\}$ be the time index that satisfies Lemma \ref{lem: geq s}, i.e., 
\iffalse
for all $v_i \in \mathcal{R}$,
\begin{align*}
    \| x_i[k^*] - \hat{x} \|_2 \leq  s^*(\xi)
\end{align*}
which implies that
\fi
\begin{align*}
     \max_{v_i \in \mathcal{R}} \| x_i[k^*] - \hat{x} \|_2 \leq  s^*(\xi).
\end{align*}
Applying Proposition~\ref{prop: alg prop}, we get
\begin{align}
    \max_{v_i \in \mathcal{R}} \| z_i[k^*] - \hat{x} \|_2  \leq \max_{v_i \in \mathcal{R}} \| x_i[k^*] - \hat{x} \|_2 \leq  s^*(\xi).  \label{eqn: max_z}
\end{align}
On the other hand, from Lemma \ref{lem: leq R_i} and \ref{lem: leq R_i to s_i}, we have that for $k \geq \max \{ k_1^* ,  k_2^* \}$, if $\| z_i[k] - \hat{x} \|_2 \leq s^*(\xi)$ then $\| x_i[k+1] - \hat{x} \|_2   \leq s^*(\xi)$. Therefore, the inequality \eqref{eqn: max_z} implies that 
\begin{align*}
    \max_{v_i \in \mathcal{R}} \| x_i[k^* + 1] - \hat{x} \|_2  \leq   s^*(\xi).  
\end{align*}
Applying the above procedure recursively, we get
\iffalse
\begin{align*}
    \max_{v_i \in \mathcal{R}} \| x_i[k^* + K] - \hat{x} \|_2  \leq   s^*(\xi)  
    \quad \text{for all} \;\; K \geq 0
\end{align*}
which is equivalent to 
\fi
\begin{align*}
    \| x_i[k] - \hat{x} \|_2  \leq   s^*(\xi)  
    \quad \text{for all} \;\; k \geq k^* \quad \text{and} \;\; v_i \in \mathcal{R}.
\end{align*}
However, since we can choose $\xi$ to be arbitrary small, we have that for all $v_i \in \mathcal{R}$,
\begin{align*}
    \limsup_k \| x_i[k] - \hat{x} \|_2 &\leq s^*(0) \\
    &= \max_{v_j \in \mathcal{R}} \big\{ \max \{ \Tilde{R}_j \sec \theta_j , \Tilde{R}_j + \delta_j \} \big\}  .
\end{align*}
Furthermore, the variables $\theta_i (\epsilon)$ and $\delta_i(\epsilon)$ depend on $\epsilon$ which is defined in \eqref{def: sublevel_set} and the above inequality holds for any $\epsilon > 0$. Hence, for all $v_i \in \mathcal{R}$,
\begin{multline*}
    \limsup_k \| x_i[k] - \hat{x} \|_2 \\
    \leq \inf_{\epsilon > 0} \Big\{ \max_{v_j \in \mathcal{R}} \big\{ \max \{ \Tilde{R}_j \sec \theta_j(\epsilon) , \; \Tilde{R}_j + \delta_j(\epsilon) \} \big\} \Big\} .
\end{multline*}
\end{proof}


%%%%%%%%%%%%%%%%%%%%%%%%%%%%%%%%%%%%%%%%%%%%%%%%%%%%%%%%%%%%%%%%%%%%%%%%%%%%%%%%

%%%%%%%%%%%%%%%%%%%%%%%%%%%%%%%%%%%%%%%%%%%%%%%%%%%%%%%%%%%%%%%%%%%%%%%%%%%%%%%%

%%%%%%%%%%%%%%%%%%%%%%%%%%%%%%%%%%%%%%%%%%%%%%%%%%%%%%%%%%%%%%%%%%%%%%%%%%%%%%%%
\bibliographystyle{IEEEtran}
\bibliography{main}

\end{document}